\documentclass[11pt]{amsart}

\usepackage{amsthm, amsfonts, amssymb, amscd, rotating}
\usepackage[pagebackref,colorlinks]{hyperref}
\usepackage{tikz-cd}
\usepackage{geometry}
\usepackage{marginnote}
\usepackage{aligned-overset}
\usepackage[utf8]{inputenc}
\usepackage{xcolor}
\definecolor{darkgreen}{rgb}{0,0.5,0}

\theoremstyle{definition}
\newtheorem{ntn}{Notation}[section]
\newtheorem{dfn}[ntn]{Definition}
\theoremstyle{plain}
\newtheorem{lem}[ntn]{Lemma}
\newtheorem{prp}[ntn]{Proposition}
\newtheorem{thm}[ntn]{Theorem}
\newtheorem{introthm}{Theorem}
\newtheorem{introcor}{Corollary}
\newtheorem{introconj}{Conjecture}

\newtheorem{cor}[ntn]{Corollary}
\newtheorem{conj}[ntn]{Conjecture}
\theoremstyle{definition}

\newtheorem{rem}[ntn]{Remark}
\newtheorem{exa}[ntn]{Example}

\numberwithin{equation}{section}

\newcommand{\z}{\mathbb{Z}}
\newcommand{\q}{\mathbb{Q}}

\newcommand{\F}{\mathbb{F}}

\newcommand{\OO}{\mathcal{O}}

\newcommand{\ppp}{\mathfrak{p}}
\newcommand{\qqq}{\mathfrak{q}}

\newcommand{\arr}{\rightarrow}
\newcommand{\larr}{\longrightarrow}
\newcommand{\harr}{\hookrightarrow}
\newcommand{\two}{\twoheadrightarrow}
\newcommand{\Gg}{\Gamma}

\newcommand{\GL}{{\rm GL}}

\newcommand{\SL}{{\rm SL}}
\newcommand{\GE}{{\rm GE}}

\renewcommand{\char}{{\rm char}}

\newcommand{\coker}{{\rm coker}}
\newcommand{\im}{{\rm im}}

\newcommand{\B}{{\rm B}}

\newcommand{\St}{{\rm St}}

\newcommand{\Spec}{{\rm Spec}}

\newtheoremstyle{athm}
{}
{}
{\itshape}
{}
{\scshape}
{}
{.5em}
{\thmnote{#3}}
\theoremstyle{athm}

\begin{document}

\title{The Schur multiplier of $\SL_2$ and Dedekind zeta-functions over $S$-integers}

\author[P. A. Alves]{Pedro H. Amorim}
\author[I. V. Picinini]{Isadora V. Picinini}
\author[B. R. Ramos]{Bruno R. Ramos}
\author[T. Verissimo]{Thiago Verissimo}

\address{\sf Instituto de Ci\^encias Matem\'aticas e de Computa\c{c}\~ao (ICMC), 
Universidade de S\~ao Paulo, S\~ao Carlos, Brasil}
\address{\sf Beijing Institute of Mathematical Sciences and Applications (BIMSA), Beijing, China}
\address{\sf Department of Mathematics, Aarhus University, Ny Munkegade 118, 8000 Aarhus C, Denmark \vspace{0.5cm}}

\email{amorim.alves@usp.br}
\email{isadoravanzella@usp.br}
\email{brramos2050@gmail.br}
\email{thiagovlg@usp.br}

\begin{abstract}
In this paper, we obtain an exact sequence connecting $H_2(\rm{SL}_2(\mathcal{O}_{K,S}), \mathbb{Z})$ to $H_2(\rm{SL}_2(\mathcal{O}_{K,T}), \mathbb{Z})$, where $\mathcal{O}_{K,S}$ is a ring of $S$-integers and $T$ is a set of primes containing $S$. We apply this sequence to establish a relation between $H_2(\rm{SL}_2(\mathcal{O}_{K,S}), \mathbb{Z})$ with the second $K$-group $K_2(\mathcal{O}_{K,S})$, for $S$ large enough. This leads to a description of the rank and size of the torsion of $H_2(\rm{SL}_2(\mathcal{O}_{K,S}), \mathbb{Z})$. As an application, we propose a homological version of the Birch-Tate formula (conjecture) under these assumptions. \\


\end{abstract}

\maketitle

\section*{Introduction} \label{int}

The (co)homology of arithmetic groups  has recently received systematic attention as it appears in many branches of mathematics as a fundamental tool to encode topological, number-theoretical and geometric data. Notably, in recent works, it was employed to investigate moments of $L$-functions \cite{MPPRW2026} and modular forms \cite{ashyasaki2021}. The explicit determination of the torsion in the homology groups is of particular interest, as it provides congruences associated to Eisenstein's series \cite{deng2023}. Furthermore, in \cite{rahm2025}, homology of linear groups is used to connect modular form spaces and algebraic $K$-theory in the framework of the Calegari-Venkatesh conjecture. In the present article, we explore the connections concerning Dedekind zeta functions.

Despite the wide range of applications, computing these homology groups explicitly is not an easy task. In fact, finding the structure of the groups $H_d(\SL_n(\OO_{K,S}), \z)$, where $\OO_{K,S}$ is a ring of $S$-integers, is still an open problem in general. Progress in this line was made in classical homological stability by Charney \cite{charney1980}, and by van der Kallen \cite{kallen1980}. Nevertheless, the unstable range remains a tough challenge. In \cite{BBT2025-2}, $H_1(\SL_2(\OO_{K,S}), \z)$ was described for $\OO_{K,S}$ having infinitely many units. The case of the second homology is much harder, and was solved for $\OO_{K,S}=\z[1/n]$, when $n$ is a prime number \cite{an1998}, and when $2, 3, 5, 7, 13 \mid n$ \cite{BBT2025-1}.

The main goal and motivation of this article is to investigate the structure of the group $H_2(\SL_2(\OO_{K,S}), \z)$. We heavily build on Hutchinson's work \cite{h2016} to explore the connections between the second integral homology of $\SL_2$ over a ring of $S$-integers and the group $K_2$. Here, we always assume $\OO_{K,S}$ is not totally imaginary and has infinitely many units. We denote by $\mathcal{S}_2$ the set of all prime ideals in $\OO_K$ with residue field $\F_2$. A refined version of Hutchinson's main theorem \cite[Theorem 6.10]{h2016} states that:

\begin{introthm}[\ref{Main2}] \label{main2-intro}
   Given a global field $K$, there exists a finite set of primes $S_0 \supseteq \mathcal{S}_2$, such that, for all $S \supseteq S_0$, there is a short exact sequence
    \[ 0 \rightarrow H_2(\SL_2(\OO_{K,S}), \z) \rightarrow H_2(\SL_2(K), \z) \rightarrow \bigoplus_{\ppp \notin S} \kappa(\ppp)^{\times} \rightarrow 0.\]
\end{introthm}

Arithmetic groups, like $\SL_2$ over $S$-integers, enjoy the pleasant property of possessing finitely generated abelian homology groups, and Theorem \ref{main2-intro} gives information of the free and torsion part of $H_2(\SL_2(\OO_{K,S}), \z)$ in terms of the number of real embeddings of $K$ and the group $K_2(\OO_{K,S})$. 

\begin{introthm} [\ref{torsionH_2}]\label{torsionH2-intro}
    Given a global field $K$, with $r$ real embeddings, there exists a finite set of primes $S_0 \supseteq \mathcal{S}_2$ such that, for all $S \supseteq S_0$,
    \begin{enumerate}
        \item[(i)] $\operatorname{rk}(H_2(\SL_2(\OO_{K,S})), \z)=r$;
        \item[(ii)] $|H_2(\SL_2(\OO_{K,S}), \z)_{\operatorname{tor}}| = |K_2(\OO_{K,S})|/{2^r}$.
    \end{enumerate}
\end{introthm}

 With the above theorem, we were able to connect the data about the structure of $H_2(\SL_2(\OO_{K,S}), \z)$ with the special value of the Dedekind zeta function $\zeta_K^S(-1)$, by means of the Birch-Tate formula for totally real abelian extensions.

\begin{introthm}[\ref{BThomthm}]
 Let $\OO_{K,S}$ be a ring of $S$-integers of a totally real abelian extension $K$, then there is $S_0 \supseteq \mathcal{S}_2$, such that for all $S \supseteq S_0$, we have, up to a power of $2$,
    \[
    |\zeta_K^S(-1)|=\frac{2^r \cdot |H_2(\SL_2(\OO_{K,S}), \z)_{\operatorname{tor}}|}{|w_2(K)|},
    \]
    where $r=[K: \q]$.  
\end{introthm}

For non-abelian extensions, Birch-Tate formula is still a conjecture and we can also give a homological version of this conjecture.

\begin{introconj}[\ref{BThomconj}]
    Let $\OO_{K,S}$ be a ring of $S$-integers of a totally real extension $K$, then there is an $S_0 \supseteq \mathcal{S}_2$, such that for all $S \supseteq S_0$, we have, up to a power of $2$:
    \[
    |\zeta_K^S(-1)|=\frac{2^r \cdot |H_2(\SL_2(\OO_{K,S}), \z)_{\operatorname{tor}}|}{|w_2(K)|},
    \]
    where $r=[K: \q]$.  
\end{introconj}

This suggests the structure of $H_2(\SL_2(\OO_{K,S}), \z)$ can be understood by looking to special values of Dedekind $L$-functions. 

Theorems \ref{main2-intro} and \ref{torsionH2-intro} are a direct consequence of the main result of this paper, generalizing a cornerstone of Hutchinson's main theorem: the exact sequence of \cite[Theorem 5.17]{h2016}, which describes how localizing more primes affects the second homology of $\SL_2(\OO_{K,S})$, provided that there exists a unit $\lambda \in \OO_{K,S}^{\times}$ such that $\lambda^2 - 1$ is also a unit. Note that if $\OO_{K,S}$ has a prime ideal $\ppp$ with $|\OO_{K,S}/\ppp| =2$, then it is not possible to have this condition on the units, as it would imply $\lambda$ and $\lambda^2 - 1$ are congruent to $1$ $(\operatorname{mod} \ppp)$, and this yields $0 = 1$ $(\operatorname{mod} \ppp)$. Analogously, the same holds for $|\OO_{K,S}/\ppp|=3$. Thus, this condition forces the assumption that $S$ contains all primes with residue field of size $2$ and $3$.

Since it is known that those primes are the ones that determine the abelianization of $\SL_2(\OO_{K,S})$ \cite{BBT2025-2}, Hutchinson's result contemplates only rings of $S$-integers for which $\SL_2(\OO_{K,S})$ is necessarily perfect. Using recent results on the first homology of congruence subgroups \cite{PIBT} and examining the Mayer-Vietoris exact sequence associated to an amalgamated decomposition of $\SL_2(\OO_{K,S})$, we generalized Hutchinson's exact sequence, proving that we only need to assume that the primes with residue field $\F_{2}$ are contained in $S$ (such condition is equivalent to $2 \nmid |H_1(\SL_2(\OO_{K,S}), \z)|$, by Remark \ref{2inS}).

\begin{introthm}[\ref{Main}]
    If $2 \nmid |H_1(\SL_2(\OO_{K,S}), \z)|$ and $T$ is a set of primes of $K$ such that $T \supsetneq S $, then we have the following exact sequence: 
    \[
    H_2(\SL_2(\mathcal{O}_{K,S}), \z) \rightarrow H_2(\SL_2(\mathcal{O}_{K, T}), \z) \rightarrow \bigoplus_{\ppp \in T\backslash S} \kappa(\ppp)^{\times} \rightarrow 0.
    \]    
\end{introthm}
\begin{introcor}[\ref{OKS_OKK}]
   If $2 \nmid |H_1(\SL_2(\OO_{K,S}), \z)|$, then there is a natural exact sequence: 
    \[
    H_2(\SL_2(\mathcal{O}_{K,S}),\z) \rightarrow H_2(\SL_2(K),\z) \rightarrow \bigoplus_{\ppp \notin S} \kappa(\ppp)^{\times} \rightarrow 0.
    \]
\end{introcor}



\textbf{Acknowledgments.} We thank Behrooz Mirzaii for his helpful guidance, comments and discussions about this article. The contributions of the first, second and fourth authors to this work were made possible by CAPES (Coordena\c{c}\~ao de Aperfeiçoamento de Pessoal de N\'ivel Superior) fellowships 
(grant numbers 88887.136061/2025-00, 88887.995616/2024-00 and 88887.955905/2024-00). The third author is supported by a research grant (VIL54509) from VILLUM FONDEN. The first author was a short-term visiting scholar at the Beijing Institute of Mathematical Sciences and Applications during the production of this article, supported by the China Latin-America Mathematics Center. 

\section{The abelianization of congruence subgroups of \texorpdfstring{$\SL_2$}{Lg} over \texorpdfstring{$S$}{Lg}-integers} \label{sec1}

This section is dedicated to presenting previous results on the abelianization of congruence subgroups of $\SL_2$ over $S$-integers. For a detailed discussion of $S$-integers, congruence subgroups, and $\GE_2$-rings, we refer the reader to \cite[\S 1, \S 2]{BBT2025-2} and \cite[\S 1]{PIBT}.

A global field $K$ is either a finite field extension of $\q$ (an algebraic number field) or a finite field extension of $\F_q(t)$. If $\char (K) = 0$ (resp. $\char (K) > 0$), then the ring of algebraic integers of $K$, denoted by $\OO_K$, is the subring of all elements of $K$ that are integral over $\z$ (resp. over $\F_q[t]$). It is well established in the literature that $\OO_K$ is a Dedekind domain \cite[Theorem 6.24]{keune2023}. 

\begin{dfn}
Let $K$ be a global field and let $S$ be a finite nonempty set of primes containing the set $S_\infty$ of all infinite primes. We define:
$$
\mathcal{O}_{K,S}=\{a \in K^{\times} \text{ }| \text{ }v_{\ppp}(a)\geq 0 \text{ for all } 
\ppp \notin S \}. 
$$
A ring $A$ is said to be a \textit{Dedekind domain of arithmetic type}, or a \textit{ring of $S$-integers}, whenever $A=\OO_{K,S}$. 
If $S = S_{\infty}$ and every prime in $S$ comes from an imaginary embedding $K \hookrightarrow \mathbb{C}$, we say that $A$ is \textit{totally imaginary}.
\end{dfn} 

Henceforth in the present work, we consider that $\mathcal{O}_{K,S}$ has infinitely many units and is not totally imaginary, unless otherwise stated. These conditions are used to ensure that $\SL_2(\OO_{K,S})$ has the Congruence Subgroup Property \cite{serre1970} and is needed for many of the previous results available in the literature.

 A well-known fact from number theory is that every nonzero prime ideal $\ppp \in \Spec(\OO_{K,S})$ is maximal (see \cite{keune2023}). Hence, $\kappa(\ppp):=A/\ppp$ is a field, the so called residue field of $A$ module $\ppp$. Moreover, if $A=\mathcal{O}_{K,S}$, then the residue field $\kappa(\ppp)$ is finite \cite[Lemma 1.2]{BBT2025-2}.

Now, consider the following set of prime ideals of a ring of integers of a global field $K$:
\[
S_2=\{\ppp\in \Spec(\OO_K): \ppp\mid 2, e_\ppp=1,  [\OO_K/\ppp: \F_2]=1\},
\]
\[
S_2'=\{\ppp\in \Spec(\OO_K): \ppp\mid 2, e_\ppp>1, [\OO_K/\ppp: \F_2]=1\},
\]
\[
\!\!\!\!\!\!\!\!\!\!\!\!\!\!\!\!\!\!
S_3=\{\ppp\in \Spec(\OO_K): \ppp \mid 3, [\OO_K/\ppp: \F_3]=1\}.
\] 

If $K$ is a finite extension of $\F_q(t)$, we may also define
\[
S_2''=\{\ppp\in\Spec(\OO_K): \ppp \mid t-a\ \text{, for some $a\in \F_2$, and}\ [\OO_K/\ppp:\F_2]=1\},
\]
for the case $q = 2$, and
\[
S_3''=\{\ppp\in\Spec(\OO_K): \ppp \mid t-a\ \text{, for some $a\in \F_3$, and}\ [\OO_K/\ppp:\F_3]=1\},
\]
for $q = 3$.

The computation of $H_1(\SL_2(\OO_{K,S}), \z)$ can be achieved in terms of the aforementioned sets of prime ideals, as we can see in the following propositions.

\begin{prp}{\cite[Theorem 3.1]{BBT2025-2}}\label{H1CHAR0}
Suppose $\char(K) = 0$ and $A = \mathcal{O}_{K,S}$. Then
\[
H_1(\SL_2(A), \z) \simeq \bigoplus_{\qqq\in S_2\backslash S} \z/4 \oplus 
\bigoplus_{\qqq\in S_2'\backslash S} (\z/2\oplus \z/2) \oplus
\bigoplus_{\qqq\in S_3\backslash S} \z/3.
\]
\end{prp}

\begin{prp}{\cite[Theorem 3.2]{BBT2025-2}}\label{H1CHAR>0}
Suppose $K$ is a finite extension of $\mathbb{F}_q(t)$ and $A = \mathcal{O}_{K,S}$. Then

\[
H_1(\SL_2(A),\z) \simeq \begin{cases}
\bigoplus_{\qqq\in S_2''\backslash S}(\z/2\oplus \z/2) &  \text{if $q=2$}\\
\bigoplus_{\qqq\in S_3''\backslash S} \z/3  & \text{if $q=3$.}\\
0 & \text{if $q\geq 4$}
\end{cases}
\]
\end{prp}

An immediate consequence is the following useful fact.

\begin{cor}
    $2 \mid |H_1(\SL_2(\OO_{K,S}), \z)|$ if and only if there exists some prime $\ppp \mid (2)$ such that $|\kappa(\ppp)|=2$ and $\ppp \notin S$.
\end{cor}

\begin{rem}
\label{2inS}
    We will be mostly dealing with the case where $2\nmid |H_1(\SL_2(\OO_{K,S}), \z)|$, i.e. when all primes $\ppp$ over $2$, with $|\kappa(\ppp)|=2$, lie on $S$. In view of this, consider the set: 
    \begin{align}
    \label{S_2}
        \mathcal{S}_2 = \begin{cases}
        S_2 \cup S_2', &\text{if $\char(K)=0$}\\
        S_2^{''}, &\text{if $\char(K)>0$}
    \end{cases}.
    \end{align}
    
\end{rem}

With Propositions \ref{H1CHAR0} and \ref{H1CHAR>0} in hand, the Congruence Subgroup Property \cite{serre1970} can be used to determine the abelianization of some congruence subgroups of $\SL_2(\OO_{K,S})$ in terms of the group $H_1(\SL_2(\OO_{K,S}), \z)$, under the same assumption on the number of units (for more details, see \cite{PIBT}). In particular, for the group
\[
\Gamma_0(\OO_{K,S}, \ppp) := \bigg\{\begin{pmatrix}
        a & b \\ c & d
    \end{pmatrix} \in \SL_2(\OO_{KS}) :  c \in \ppp \bigg\}, 
\]
we have:

\begin{thm}{\cite[Theorem 3.6]{PIBT}}\label{Gg_0Glob}
    Let $A=\OO_{K,S}$ and let $(0) \neq \ppp \in \Spec(A)$.
    \begin{itemize}
        \item [(i)] If $|\kappa(\ppp)| \geq 4$, then
        \[
        H_1(\Gg_0(A, \ppp), \z) \simeq \kappa(\ppp)^\times \oplus H_1(\SL_2(A), \mathbb{Z}).
        \]
        \item [(ii)] If $|\kappa(\ppp)|=3$, then
        \[
        H_1(\Gg_0(A, \ppp), \z)\simeq \kappa(\ppp) \oplus \kappa(\ppp)^\times \oplus H_1(\SL_2(A), \z).  
        \]
    \end{itemize}
\end{thm}

Even though $H_1(\SL_2(\OO_{K,S}), \z)$ is completely determined by Propositions \ref{H1CHAR0} and \ref{H1CHAR>0}, few information about $H_2(\SL_2(\OO_{K,S}), \z)$ is known. Notably, this group is finitely generated \cite[Lemmas 3.14 and 3.15]{PIBT}, so we may write
\[
       H_2(\SL_2(A), \z) \simeq \z^r \oplus H_2(\SL_2(A), \z)_{\operatorname{tor}},
\]
 where the left and right summands are the free and torsion parts of $H_2(\SL_2(A), \z)$, respectively. In next sections, we aim to provide estimatives for the rank and for the torsion part of this group.  




\section{An exact sequence for \texorpdfstring{$H_2(\SL_2(\OO_{K,S}), \z)$}{Lg}} 

The purpose of the present section is to construct an exact sequence relating the groups $H_2(\SL_2(\OO_{K,S}), \z)$ and  $H_2(\SL_2(\OO_{K,T}), \z)$, where $T \supsetneq S$. This composes the main theorem of the current work, generalizing \cite[Theorem 6.6]{BBT2025-1} and \cite[Corollary 5.20]{h2016}.

Let $\pi$ be a uniformizer of the global field $K$ relative to the $\ppp$-adic valuation \cite[$\S$ 10.32]{keune2023} and consider the injective maps: 
\begin{align*}
i_1: \Gg_0(\OO_{K,S}, \ppp) &\hookrightarrow \SL_2(\OO_{K,S}) &  & &j_1:\SL_2(\OO_{K,S}) &\hookrightarrow \SL_2(\OO_{K,S})\\
\begin{pmatrix}
    a & b \\
    c & d
\end{pmatrix}
&\mapsto 
\begin{pmatrix}
    a & b \\
    c & d
\end{pmatrix};
& & 
&\begin{pmatrix}
    a & b \\
    c & d
\end{pmatrix}
&\mapsto 
\begin{pmatrix}
    a & b \\
    c & d
\end{pmatrix};
\end{align*}
\begin{align*}
i_2: \Gg_0(\OO_{K,S}, \ppp) &\hookrightarrow \SL_2(\OO_{K,S}) &  & &j_2:\SL_2(\OO_{K,S}) &\hookrightarrow \SL_2(\OO_{K,S})\\
\begin{pmatrix}
    a & b \\
    c & d
\end{pmatrix}
&\mapsto 
\begin{pmatrix}
    a & \pi b \\
   \pi^{-1} c & d
\end{pmatrix};
& & 
&\begin{pmatrix}
    a & b \\
    c & d
\end{pmatrix}
&\mapsto 
\begin{pmatrix}
    a & \pi^{-1} b \\
    \pi c & d
\end{pmatrix}.
\end{align*}

The following amalgamated product is proved in \cite{serre1980}, and follows from Serre's Theory of Trees .

\begin{thm}\label{amalg}
Let $A=\OO_{K,S}$ and $\ppp$ a nonzero prime ideal of $A$. Then
\[
\SL_2(\OO_{K, S\cup \{\ppp\}}) \simeq \SL_2(\OO_{K,S}) \ast_{\Gamma_0(\OO_{K,S}, \ppp)} \SL_2(\OO_{K,S}).
\]
\end{thm}
\begin{proof}
  See \cite[p. 80, Example b)]{serre1980}.  
\end{proof}

This decomposition induces the Mayer-Vietoris exact sequence \cite[Corollary 7.7, \S 7, Chap. II]{brown1994}:

\begin{align*}
\begin{array}{c}
H_2(\Gamma_0(\OO_{K, S}, \ppp),\z) \overset{\alpha_2}{\larr} H_2(\SL_2(\OO_{K, S}),\z) 
\oplus H_2(\SL_2(\OO_{K,S}),\z) 
\overset{\beta_2}{\larr}  H_2(\SL_2(\OO_{K, S\cup \{\ppp\}}),\z)\\
\\
\larr H_1(\Gamma_0(\OO_{K, S}, \ppp),\z) \overset{\alpha_1}{\larr} H_1(\SL_2(\OO_{K, S}),\z)
\oplus H_1(\SL_2(\OO_{K,S}),\z)\\
\\
\overset{\beta_1}{\larr} H_1(\SL_2(\OO_{K, S\cup \{\ppp\}}),\z) \arr 0,
\end{array}
\end{align*}

\noindent where, for $k=1,2$, $\alpha_k(x)=({i_1}_\ast(x), {i_2}_\ast(x))$ and $\beta_k(y, z)={j_2}_\ast(z)-{j_1}_\ast(y)$, with ${i_1}_\ast, {i_2}_\ast$ induced by the maps $i_1, i_2$, and $ {j_1}_\ast, {j_2}_\ast$ induced by the maps $j_1, j_2$. In light of this, we are driven to address the homologies of the groups appearing in Theorem \ref{amalg}.

Now, given a finite abelian group $G$, we denote by $G_{(p)}$ its Sylow $p$-group. Moreover, for any field $F$, we consider the Borel subgroups of $\SL_2(F)$ defined as follows:
\[
B(F) = \bigg\{\begin{pmatrix}
        a & b \\ 0 & a^{-1}
    \end{pmatrix} \mid a, b \in F 
    \bigg\} \ \ \ \text{ and } \ \ \ 
B'(F) = \bigg\{\begin{pmatrix}
        a & 0 \\ c & a^{-1}
    \end{pmatrix} \mid a, c \in F 
    \bigg\}.
\]
Clearly, $B(F)$ and $B'(F)$ are isomorphic.

\begin{lem}
\label{pgroupiso}
    Let $k$ be any finite field of characteristic $p$ and let $M$ be an $\SL_2(k)$-module. Then, for all $i \geq 1$, the natural inclusion $B(k) \harr \SL_2(k)$ induces an isomorphism
    \[
    H_i(B(k), M)_{(p)} \simeq H_i(\SL_2(k), M)_{(p)}.
    \]
\end{lem}
\begin{proof}
    See \cite[Corollary 3.10.2]{h2013}. 
\end{proof}

\begin{rem} 
    Since $B(k)$ and $\SL_2(k)$ are finite groups, their homologies $H_i(B(k), M)$ and $H_i(\SL_2(k), M)$ are also finite \cite[Corollary 10.2, \S 10, Chap. III]{brown1994}.
\end{rem}

Let $1 \arr H \arr G \arr G/H \arr 1$ be an extension. 
The Lyndon/Hochschild-Serre spectral sequence \cite[Theorem 6.3, \S 6, Chap. VII]{brown1994} associated to this extension is the first quadrant spectral sequence 
\[
E^2_{p,q} = H_p(G/H, H_q(H, \z)) \implies H_{p+q}(G, \z).
\]
The next lemma will be used in what follows.

\begin{lem} \label{Lambda(G)}
    Let $\Lambda(G)$ be the cokernel of the natural map $H_2(H, \z) \arr H_2(G, \z)$, i.e. the sequence 
    \[
    H_2(H, \z) \arr H_2(G, \z) \arr \Lambda(G)\arr 0
    \]
    is exact. Then $\Lambda(G)$ fits into the exact sequence
    \[
    0 \arr E^{\infty}_{1,1} \arr \Lambda(G) \arr E^{\infty}_{2,0} \arr 0.
    \]
\end{lem}
\begin{proof}
    The spectral sequence gives the filtration 
    \[
    0 = F_{-1}H_2(G,\z) \subseteq F_0H_2(G,\z) \subseteq F_1H_2(G,\z) \subseteq F_2H_2(G,\z)= H_2(G,\z).
    \]
    A careful analysis of this spectral sequence gives the following short exact sequences:
    \[
    0 \arr E^{\infty}_{0,2} \arr F_1H_2(G,\z) \arr E^{\infty}_{1,1} \arr 0
    \]
    and 
    \[
    0 \arr F_1H_2(G,\z) \arr H_2(G, \z) \arr E^{\infty}_{2,0} \arr 0.
    \]
    Moreover, we also obtain the surjective edge map:
    \[
    \lambda: H_2(H, \z) \two H_2(H, \z)_{G/H} \simeq E^2_{0,2} \two E^{\infty}_{0,2} \simeq F_0H_2(G,\z).
    \]
    Let $\gamma$ be the following composition
    \[
     H_2(H,\z) \overset{\lambda}{\two} F_0H_2(G,\z) \hookrightarrow H_2(G, \z).
    \]
    Then $\Lambda(G) = \coker \ \gamma$. Thus, we then have the exact sequence
    \[
    H_2(H, \z) \overset{\gamma}{\arr} H_2(G, \z) \arr \Lambda(G) \arr 0.
    \]
    Now, let $\beta$ be the composition
    \[
    H_2(H,\z) \overset{\lambda}{\two} F_0H_2(G,\z) \hookrightarrow F_1H_2(G,\z).
    \]
    Then we have the following commutative diagram with exact rows:
    \[
    \begin{tikzcd}
    & H_2(H, \z) \ar[r, "\gamma"]\ar[d, "\beta"]& H_2(G, \z) \ar[r]\ar[d, equal]& \Lambda(G) 
    \ar[r] \ar[d, "\alpha"]& 0\\
    0 \ar[r] & F_1H_2(G,\z) \ar[r]& H_2(G, \z) \ar[r]& E^{\infty}_{2,0} \ar[r]& 0,
    \end{tikzcd}
    \]
    where the map $\alpha$ is surjective and is induced as follows:
    \[
    \Lambda(G) \simeq \frac{H_2(G,\z)}{\rm im \ \gamma} = \frac{H_2(G,\z)}{F_0H_2(G,\z)} \two \frac{H_2(G,\z)}{F_1H_2(G,\z)} \simeq E^{\infty}_{2,0}.
    \]
    By applying the Snake lemma to the above diagram, we get $\rm ker \ \alpha = \rm coker \ \beta$. But $\textrm{coker} \ \beta = E^{\infty}_{1,1}$. Putting all these together, we obtain the required short exact sequence:
    \[
    0 \arr E^{\infty}_{1,1} \arr \Lambda(G) \overset{\alpha}{\arr} E^{\infty}_{2,0} \arr 0.
    \]
\end{proof}

The next result addresses a crucial technical fact on the second integral homology of $\Gamma_0(A,\ppp)$ and is derived from a combination of the techniques present in \cite[Proposition 5.16]{h2016} and \cite[Theorem 5.2]{BBT2025-1}.

\begin{prp} \label{surj}
   Let $A=\OO_{K,S}$ and let $\ppp$ be a nonzero prime ideal of $A$ such that $|\kappa(\ppp)|\neq 2$. Then, the natural maps 
    \[
    i_{1\ast}: H_2(\Gamma_0(A, \ppp), \z) \arr H_2(\SL_2(A), \z) \ \text{ and } \ i_{2\ast}: H_2(\Gamma_0(A, \ppp), \z) \arr H_2(\SL_2(A), \z)
    \]
    are surjective.
\end{prp}
\begin{proof}
     We will show that $i_{1\ast}$ is surjective and, using analogous arguments, one can conclude the same for $i_{2\ast}$. For this, we split the proof in two cases: when $|\kappa(\ppp)| \neq 3$ and when $|\kappa(\ppp)| = 3$. 

     \vspace{0.3cm}
     (i) $|\kappa(\ppp)|\neq 3$: 
     \vspace{0.3cm}

     Consider the following commutative diagram with exact rows: 
\begin{equation*}
\begin{tikzcd}
1 \ar[r] & \Gamma(A,\ppp) \ar[r]\ar[d, equal]& \Gamma_0(A,\ppp) \ar[r]\ar[d, "i_1"]& \B(\kappa(\ppp)) 
\ar[r] \ar[d, hook]& 1\\
1 \ar[r] & \Gamma(A,\ppp) \ar[r]& \SL_2(A) \ar[r]& \SL_2(\kappa(\ppp)) \ar[r]& 1.
\end{tikzcd}
\end{equation*}

From the top exact sequence we obtain a spectral sequence on the second page: 
\[
E_{i,j}^2(\Gg_0(A, \ppp))=H_i(B(\kappa(\ppp), H_1(\Gg(A, \ppp), \z)) \implies H_{i+j}(\Gg_0(A, \ppp), \z),
\]
while the bottom one gives us: 
\[
E_{i,j}^2(\SL_2(A))=H_i(\SL_2(\kappa(\ppp)), H_1(\Gg(A, \ppp), \z)) \implies H_{i+j}(\SL_2(A), \z).
\]
The image of the edge morphisms: $E_{0, j}^{\infty}(H) \rightarrow H_j(H,\z)$, coincides with the image of the maps $H_j(\Gg(A, \ppp), \z) \rightarrow H_j(H, \z)$, for $H\in \{\Gg_0(A, \ppp), \SL_2(A)\}$. Hence, by analysis on the the $E^\infty$-terms of degree $2$ in the extension, we obtain a commutative digram of the form: 
     \begin{equation}
\begin{tikzcd}
\label{diag}
 H_2(\Gamma(A,\ppp), \z) \ar[r]\ar[d, equal]& H_2(\Gamma_0(A,\ppp), \z) \ar[r]\ar[d, "i_{1\ast}"]& \Lambda(\Gg(A, \ppp)) 
\ar[r] \ar[d]& 0\\
H_2(\Gamma(A,\ppp), \z) \ar[r]& H_2(\SL_2(A), \z) \ar[r]& \Lambda(\SL_2(A)) \ar[r]& 0,
\end{tikzcd}
\end{equation}
where $\Lambda(H)$ is a group fitting into the exact sequence: 
\begin{equation}
\label{exactlamb}
0 \rightarrow  E_{1,1}^\infty(H) \rightarrow \Lambda(H) \rightarrow E_{2,0}^\infty(H) \rightarrow 0
\end{equation}
(see Lemma \ref{Lambda(G)}). By diagram chasing, it is enough to prove that the map $\Lambda(\Gg(A, \ppp)) \rightarrow \Lambda(\SL_2(A))$ is surjective. 
By \cite[Lemma 2.9]{PIBT}, we have that:
\[
H_1(\Gg(A, \ppp), \z) \simeq G \oplus H_1(\SL_2(A), \z),
\]
where $G$ is a $p$-group.  By Propositions \ref{H1CHAR0} and \ref{H1CHAR>0}, $|H_1(\SL_2(A), \z)| = 2^k3^s$, for some $k, s \in \z$. If $\char(\kappa(\ppp)) \neq 2$, the same propositions give $H_1(\Gg(A, \ppp), \z)_{(p)} =G$. Otherwise, when $\char(\kappa(\ppp))=2$, we have $H_1(\Gg(A, \ppp), \z)_{(p)}=G \oplus (\z/4)^{S_2} \oplus (\z/2 \oplus \z/2)^{S_2'}$. Either way, we may write $H_1(\Gg(A, \ppp), \z)=\mathcal{K} \oplus \mathcal{G}$, where $\mathcal{K}=H_1(\Gg(A, \ppp), \z)_{(p)}$.

Now, 
\begin{align*}
E_{i, 1}^2(\Gg_0(A, \ppp))=&H_i(B(\kappa(\ppp), H_1(\Gg(A, \ppp), \z)) \simeq H_i(B(\kappa(\ppp)), \mathcal{K} \oplus \mathcal{G)}\\
&\simeq H_i(B(\kappa(\ppp)), \mathcal{K}) \oplus H_i(B(\kappa(\ppp)), \mathcal{G}).   
\end{align*}

Once $B(\kappa(\ppp))$ acts trivially on $H_1(\SL_2(A), \z)$, it also acts trivially on $\mathcal{K}$. Therefore, by Universal Coefficient-Theorem, 
\begin{align*}
    H_i(B(\kappa(\ppp)), \mathcal{K}) &\simeq H_i(B(\kappa(\ppp)), \z) \otimes_{\z} \mathcal{K},\\
    H_i(B(\kappa(\ppp)), \mathcal{G}) &\simeq H_i(B(\kappa(\ppp)), \z) \otimes_{\z} \mathcal{G}.
\end{align*}
By \cite[Corollary 3.10 (2)]{h2013}, we have that $H_i(B(\kappa(\ppp), \z) \otimes_{\z} \mathcal{G}\simeq 0$ for $|\kappa(\ppp)| \neq 2, 3$. This implies that $E_{i,1}^2(\Gg_0(A, \ppp)) \simeq H_i(B(\kappa(\ppp)), \mathcal{K})=H_i(B(\kappa(\ppp)), \mathcal{K})_{(p)}$ and by similar arguments we can show that $E_{i, 1}^2 (\SL_2(A)) \simeq H_i(\SL_2(\kappa(\ppp)), \mathcal{K})_{(p)}$. As a consequence, Lemma~\ref{pgroupiso} ensures that
\[
E_{i,0}^2(\Gg_0(A, \ppp)) \simeq E_{i,0}^2(\SL_2(A)). 
\]
Evoking Lemma \ref{pgroupiso} again we obtain: 
\[
E_{i,0}^2(\Gg_0(A, \ppp)) =H_i(B(\kappa(\ppp), \z)_{(p)} \simeq H_i(\SL_2(\kappa(\ppp)), \z)_{(p)}=E_{i, 0}^2(\SL_2(A)). 
\]
On the other hand, since $E_{1, 1}^\infty=\operatorname{ker}(d_{1,1}^2: E_{3,0}^2 \rightarrow E_{1, 1}^2)$, and $E_{i, j}^2(\Gg_0(A, \ppp)) \simeq  E_{i, j}^2(\SL_2(A))$ for $i \in \{1, 3\}$ and $j \in \{0, 1\}$, we have 
\[
E_{1,1}^\infty (\SL_2(A)) \simeq E_{1, 1}^\infty(\Gg_0(A, \ppp)).
\]
Moreover, $E_{2,0}^\infty=\ker(d_{2,0}^2: E_{2,0}^2 \rightarrow E_{0,1}^2)$ and, by the direct computations from \cite[$\S$ 3]{h2013}, we get: 
\[
H_2(B(\kappa(\ppp)), \z)=H_2(B(\kappa(\ppp)), \z)_{(p)}\simeq H_2(\SL_2(\kappa(\ppp), \z)_{(p)}=H_2(\SL_2(\kappa(\ppp)), \z),
\]
so 
\[
E_{2,0}^\infty(\SL_2(A)) \simeq E_{2,0}^{\infty} (\Gg_0(A, \ppp)). 
\]

By the sequence (\ref{exactlamb}), we obtain $\Lambda(\SL_2(A))\simeq \Lambda(\Gg_0(A, \ppp))$. Now, diagram chasing applied on diagram (\ref{diag}) concludes the proof for this case.
\vspace{0.3cm}

(ii) $|\kappa(\ppp)|=3$: 
\vspace{0.3cm}

In this case $B(\kappa(\ppp))=\F_3$, so the proof can be obtained as in \cite[Theorem 5.2]{BBT2025-1}.

\end{proof}

\begin{lem}
\label{OK,S-OKS,P}
    Given a ring of $S$-integers $\OO_{K,S}$ and a prime $\ppp \in \Spec(\OO_{K}) \backslash S$ such that $|\kappa(\ppp)| \geq 3$, we have an exact sequence:
       \[
    H_2(\SL_2(\mathcal{O}_{K,S}), \z) \rightarrow H_2(\SL_2(\mathcal{O}_{K,S\cup \{{\ppp}\}}), \z) \rightarrow \kappa(\ppp)^\times \rightarrow 0.
    \]
\end{lem}
\begin{proof}
    We split the proof in two cases: when $|\kappa(\ppp)|> 3$ and when $|\kappa(\ppp)|=3$. 
    \vspace{0.3cm}

    (i) $|\kappa(\ppp)|> 3$: 
    \vspace{0.3cm}
    
    By Propositions \ref{H1CHAR0} and \ref{H1CHAR>0}, $H_1(\SL_2(\OO_{K,S\cup\{\ppp\}}), \z)\simeq H_1(\SL_2(\OO_{K,S}), \z)$. Combining with Proposition \ref{surj}, the Mayer-Vietoris exact sequence boils down to (integral coefficients are omitted):
    \[
        H_2(\SL_2(\OO_{K,S})) \rightarrow H_2(\SL_2(\OO_{K,S\cup\{\ppp\}})) \rightarrow H_1(\Gamma_0(\OO_{K,S}, \ppp)) \rightarrow H_1(\SL_2(\OO_{K,S})) \rightarrow 0.
    \]
     Now, from Theorem \ref{Gg_0Glob} (i), $H_1(\Gamma_0(\OO_{K, S}, \ppp), \z)\simeq H_1(\SL_2(\OO_{K,S}), \z) \oplus \kappa(\ppp)^{\times}$. Thus, replacing it in the above exact sequence, we get the required one:
    \[
       H_2(\SL_2(\mathcal{O}_{K,S}), \z) \rightarrow H_2(\SL_2(\mathcal{O}_{K,S\cup \{{\ppp}\}}), \z) \rightarrow \kappa(\ppp)^{\times} \rightarrow 0.
    \]

    (ii) $|\kappa(\ppp)|=3$: 
    \vspace{0.3cm}
    
    The proof follows in a similar vein. Propositions \ref{H1CHAR0} and \ref{H1CHAR>0} give the isomorphism $H_1(\SL_2(\OO_{K,S\cup \{\ppp\}}), \z)\simeq H_1(\SL_2(\OO_{K,S}), \z)\oplus \kappa(\ppp)$. Applying Proposition \ref{surj}, we obtain (integral coefficients again omitted):
    \[
            H_2(\SL_2(\OO_{K,S})) \rightarrow H_2(\SL_2(\OO_{K,S\cup\{\ppp\}})) \rightarrow H_1(\Gamma_0(\OO_{K,S}, \ppp)) 
      \rightarrow H_1(\SL_2(\OO_{K,S}))\oplus \kappa(\ppp) \rightarrow 0.
    \]
    
    Finally, Theorem \ref{Gg_0Glob} (ii) gives $H_1(\Gamma_0(\OO_{K,S}, \ppp),\z)\simeq H_1(\SL_2(\OO_{K, S}), \z)\oplus \kappa(\ppp) \oplus \kappa(\ppp)^\times$. Replacing it in the above exact sequence we get:
    \[
       H_2(\SL_2(\mathcal{O}_{K,S}), \z) \rightarrow H_2(\SL_2(\mathcal{O}_{K,S\cup \{{\ppp}\}}), \z) \rightarrow \kappa(\ppp)^{\times} \rightarrow 0.
    \]
\end{proof}

Inspired by the previous lemma and proceeding as in \cite[\S 6]{BBT2025-1}, we denote by $\delta_\ppp$ the surjective map  $H_2(\SL_2(\OO_{K,S\cup\{\ppp\}}), \z) \twoheadrightarrow \kappa(\ppp)^\times$. Given a finite set of places $T$ containing $S$ properly, this induces the following natural map:
\[
\delta_T: H_2(\SL_2(\OO_{K,S \cup T}), \z) \rightarrow \bigoplus_{\ppp \in T \backslash S} \kappa(\ppp)^\times, \ \text{ where } \ [X] \mapsto \bigoplus_{\ppp \in T\backslash S} \delta_{\ppp}([X]).
\]

We remark that this map $\delta_{\ppp}$ is essentially the tame symbol of $K$-theory (see \cite[$\S$ 6]{BBT2025-1} and \cite[Theorem 5.17]{h2016}).

The next theorem is the main result of the present work and will be of great importance for understanding the structure $H_2(\SL_2(\OO_{K,S}), \z)$. 

\begin{thm} \label{Main} 
    If $2 \nmid |H_1(\SL_2(\OO_{K,S}), \z)|$ and $T$ is a set of primes of $K$ such that $T \supsetneq S $, then we have the following exact sequence: 
    \[
    H_2(\SL_2(\mathcal{O}_{K,S}), \z) \rightarrow H_2(\SL_2(\mathcal{O}_{K, T}), \z) \rightarrow \bigoplus_{\ppp \in T\backslash S} \kappa(\ppp)^{\times} \rightarrow 0.
    \]    
\end{thm}
\begin{proof}
   We proceed by induction on $|T \backslash S|$. The base case $|T \backslash S|=1$ is Lemma \ref{OK,S-OKS,P}. Now, let $|T \backslash S|>1$. Given $\qqq \in T\backslash S$, we consider $T'=T \backslash \{\qqq\}$. By Lemma \ref{OK,S-OKS,P}, we obtain an exact sequence: 
    \[
    H_2(\SL_2(\mathcal{O}_{K, T'}), \z) \rightarrow H_2(\SL_2(\mathcal{O}_{K,T}), \z) \rightarrow \kappa(\qqq)^\times \rightarrow 0.
    \]
    Since $|T' \backslash S|< |T \backslash S|$, the induction hypothesis yields an exact sequence:
    \[
     H_2(\SL_2(\mathcal{O}_{K,S}), \z) \overset{\gamma}{\rightarrow} H_2(\SL_2(\mathcal{O}_{K, T^\prime}), \z) \overset{\delta_{T^{\prime}}}{\rightarrow} \bigoplus_{\ppp \in T^\prime\backslash S} \kappa(\ppp)^{\times} \rightarrow 0.
    \]    
    Now, by applying the Snake lemma to the following commutative diagram with exact rows:
    \[
\begin{tikzcd}
    & \mathrm{H}_2(\mathrm{SL}_2(\mathcal{O}_{K,T'}), \mathbb{Z}) \arrow[r] \arrow[d, "\delta_{T'}" , two heads] 
    & \mathrm{H}_2(\mathrm{SL}_2(\mathcal{O}_{K,T}), \mathbb{Z}) \arrow[r] \arrow[d, "\delta_T"] 
    & k(\mathfrak{q})^\times \arrow[r] \arrow[d, "\mathrm{id}"] 
    & 0 \\
  0 \arrow[r] 
    & \displaystyle\bigoplus_{\mathfrak{p}\in T'\backslash S} k(\mathfrak{p})^\times \arrow[r] 
    & \displaystyle\bigoplus_{\mathfrak{p}\in T \backslash S} k(\mathfrak{p})^\times \arrow[r] 
    & k(\mathfrak{q})^\times \arrow[r] 
    & 0,
\end{tikzcd}
\]  
we get that $\delta_T$ and $(\ker \delta_{T^{\prime}} \arr \ker \delta_{T})$ are surjective. Since $\ker \delta_{T^{\prime}} = \im \gamma$, $H_2(\SL_2(\mathcal{O}_{K,S}), \z)$ surjects onto $\ker \delta_{T}$, from which the result follows.
\end{proof}

\begin{rem}
    This exact sequence cannot be obtained in general without the hypothesis of all inert primes over $2$ lying on $S$. For this, see Remark \ref{Sharpness} in next section. 
\end{rem}

\begin{cor}
\label{OKS_OKK}
   If $2 \nmid |H_1(\SL_2(\OO_{K,S}), \z)|$, then there is a natural exact sequence: 
    \[
    H_2(\SL_2(\mathcal{O}_{K,S}),\z) \rightarrow H_2(\SL_2(K),\z) \rightarrow \bigoplus_{\ppp \notin S} \kappa(\ppp)^{\times} \rightarrow 0.
    \]
\end{cor}
\begin{proof}
    Since $K = \underset{S\subset T}{\rm{colim \ }} \OO_{K,T}$, the result follows immediately from Theorem \ref{Main} by taking colimits (here, the functor $\rm colim$ is just union and is right exact).
\end{proof}


\section{The second homology of \texorpdfstring{$\SL_2(\OO_{K,S})$ and the group $K_2(\OO_{K,S})$}{Lg}}

Once we have an exact sequence relating $H_2(\SL_2(\OO_{K,S}), \z)$ and $H_2(\SL_2(K), \z)$ (Corollary \ref{OKS_OKK}), we can use results from Hutchinson \cite{h2016} and Matsumoto \cite{mat1969} to relate the first group to $K_2(\OO_{K,S})$, providing information about the structure of $H_2(\SL_2(\OO_{K,S}), \z)$. In particular, for $S$ large enough, we see that the torsion part of $H_2(\SL_2(\OO_{K,S}), \z)$ is controlled by $K_2(\OO_{K,S})$ and the free part is just $\z^r$, where $r$ is the number of real embeddings of $K$ (see Theorem \ref{torsionH_2} below). 

\subsection{The group \texorpdfstring{$K_2(2, \OO_{K,S})$}{K2(2, OK,S)}} 
We start by recalling relevant facts about the second unstable $K$-group of a commutative ring $A$, which is denoted by $K_2(2,A)$.  

\begin{dfn}
    Let $A$ be a commutative ring. The second Steinberg group of $A$, $\operatorname{St}_2(A)$, is defined as the free group generated by the letters $x_{ij}(\lambda), \lambda \in A, 1 \leq i \neq j \leq 2$, and $w_{ij}(u):=x_{ij}(u)x_{ji}(-u)x_{ij}(u)$, with $u \in A^\times$, subject to the following relations:     
    \begin{itemize}
        \item [(i)] $x_{ij}(s)x_{ij}(t)=x_{ij}(s+t)$;
        \item [(ii)] $w_{ij}(u)x_{ij}(t)w_{ij}(-u)=x_{ji}(-u^{-2}t)$.
    \end{itemize}
\end{dfn}

It is clear that there is a natural group homomorphism
\begin{align*}
    \phi: \operatorname{St}(2, A) \rightarrow \SL_2(A), \ \text{ with } \ x_{ij}(\lambda) \mapsto E_{ij}(\lambda),
\end{align*}
where $E_{ij}(\lambda)$ denotes an elementary matrix with $\lambda$ on the $(i, j)$ coordinate. This leads to the definition of $K_2(2, A)$.

\begin{dfn}
    The unstable $K_2$-group of degree 2 of $A$, denoted by $K_2(2, A)$, is defined to be the kernel of $\phi$.
\end{dfn}

Let $F$ be an infinite field. A classic result due to Matsumoto and Moore relates $H_2(\SL_2(F), \z)$ to $K_2(2, F)$.

\begin{thm}{\cite{mat1969, moore1968}}\label{matso1}
    Let $F$ be an infinite field.
    \begin{itemize}
        \item [(i)] The exact sequence
    \[
    1 \rightarrow K_2(2, F) \rightarrow \St(2, F) \rightarrow \SL_2(F) \rightarrow 1
    \]
    is a universal central extension of the group $\SL_2(F)$. In particular, 
    \[
    H_2(\SL_2(F), \z) \simeq K_2(2, F).
    \]
    \item [(ii)] The group $K_2(2, F)$ is the free group generated by symbols $c(u,v)$, with $u, v \in K^\times$, subject to the following relations: 
    \begin{itemize}
        \item [(a)] $c(u,v)=1$, whenever $u=1$ or $v=1$;
        \item [(b)] $c(u, v)=c(u^{-1}, v)$;
        \item [(c)] $c(u, vw)c(v, w)=c(uv, w) c(u, v)$, with $w\in K^\times$;
        \item [(d)] $c(u,v) = c(u, -uv)$;
        \item [(e)] $c(u, v)=c(u, (1-u)v)$, with $u\neq 1$. 
    \end{itemize}
    \end{itemize}
\end{thm}

Matsumoto also describes the second stable $K$-group of any field $F$, $K_2(F)$.

\begin{thm} {\cite[$\S$ 12]{milnor1971}} \label{matso2}
    Let $F$ be a field. Then $K_2(F)$ is isomorphic to the free group generated by symbols $\{x, y\}, x, y \in F^\times$, subject to the following relations: 
    \begin{itemize}
        \item [(i)] $\{x, 1-x\}=1, x \neq 1$;
        \item [(ii)] $\{x_1x_2,y\}=\{x_1, y\}\{x_2, y\}$; 
        \item [(iii)] $\{x, y_1y_2\}=\{x, y_1\}\{x, y_2\}$.
    \end{itemize}
\end{thm}


\begin{cor}
\label{K2(2K)K2(K)}
    For any field $F$, the following homomorphism of groups is surjective: 
    \begin{align*}
    \psi: K_2(2, F) \rightarrow K_2(F), \ \text{ with } \ c(u,v) \mapsto \{u,v\}.
\end{align*}
\end{cor}

The group $K_2(2, F)$, $F$ an infinite field, is closely connected to the $K$-theory of quadratic forms. In this case, it is well established that $K_2(2, F) \simeq K_2^{MW}(F)$, the second Milnor-Witt $K$-group \cite[Proposition 3.12]{h2016}. Moreover, it fits into an exact sequence of the form
\[ 0 \rightarrow I^{3}(F) \rightarrow K_2(2, F) \rightarrow K_2(F) \rightarrow 0,\]
where $I(F)$ is the fundamental ideal associated to the Witt ring $W(F)$ \cite[Corollary 3.10]{h2016}. 

\begin{exa} \label{Z1/m}
    For the case of rational numbers, $K_2^{MW}(\q)$ is known to be $\z \oplus \underset{p \ \text{prime}}{\bigoplus} \F_p^{\times}$ \cite[Proposition 2.5]{kolderup22}. Using this information in the exact sequence of Corollary \ref{OKS_OKK} for $\OO_{K,S} = \z[1/2n]$, where $n$ is a positive odd integer, we have 
    \begin{equation}\label{seq-Q}
    H_2(\SL_2(\z[1/2n],\z) \rightarrow H_2(\SL_2(\q),\z) \rightarrow \bigoplus_{p \nmid n\text{, prime}} \F_p^{\times} \rightarrow 0,
    \end{equation}
    which becomes
    \[
    H_2(\SL_2(\z[1/2n],\z) \rightarrow \z \oplus \bigoplus_{p \ \text{prime}} \F_p^{\times} \rightarrow \bigoplus_{p \nmid n\text{, prime}} \F_p^{\times} \rightarrow 0.
    \]
    Note that, by \cite[Theorem 6.8]{BBT2025-1},
    \[H_2(\SL_2(\z[1/2n], \z) \simeq \z \oplus \bigoplus_{p \mid n\text{, prime}}\F_p^{\times}.\]
    In this case, sequence \ref{seq-Q} is actually short exact, motivating us to ask in what other cases this is true. As we will see in the next subsection, this question is associated to the study of the group $\tilde{K}_2(2, \OO_{K,S})$, defined by Hutchinson \cite[$\S$ 6]{h2016}.

    
    

\end{exa}

\begin{rem}
    With our finds and knowing that $H_2(\SL_2(\q), \z)$ is an abelian group of rank $1$ \cite[Remark 6.7]{h2016}, we can deduce that $H_2(\SL_2(\q), \z) \simeq  \z \oplus \underset{p \ \text{prime}}{\bigoplus} \F_p^{\times}$, recovering Kolderup's description of $K_2^{MW}(\q)$ \cite[Proposition 2.5]{kolderup22}. Indeed, we can write $H_2(\SL_2(\q),\z) \simeq \z \oplus H$, for some abelian group $H$ formed by a direct sum of divisible and/or reduced groups \cite[Chap. IV]{fuchs1970}. Thus, for $n=1$, sequence \ref{seq-Q} is short exact:
    \[
    0 \rightarrow \z \rightarrow \z \oplus H \rightarrow \bigoplus_{p \text{ prime}} \F_p^{\times} \rightarrow 0,
    \]
    and hence $\frac{\z \oplus H}{\z} \simeq \bigoplus_{p \text{ prime}} \F_p^{\times}$, i.e. $H \simeq \bigoplus_{p \text{ prime}} \F_p^{\times}$. 
\end{rem}

\subsection{The group \texorpdfstring{$\tilde{K}_2(2,\OO_{K,S})$}{Lg}} \label{Section 3.2}


This section is dedicated to presenting $\tilde{K}_2(2, \OO_{K,S})$, a subgroup of $K_2(2, K)$ closely related to $K_2(\OO_{K,S})$. This group is built using the so called \textit{tame symbol} $\tau_\ppp$ of $K$-theory \cite[Lemma 6.3, Chap. 3]{wei2013}. Given a prime ideal $\ppp$ of the ring of integers of a global field $K$, we define
\[
    \tau_\ppp: K_2(K) \rightarrow \kappa(\ppp)^\times, \ \text{ with } \ \{x, y\} \mapsto (-1)^{v_\ppp(x)v_\ppp(y)}x^{-v_\ppp(y)}y^{v_\ppp(x)} \mod \ppp.
\]
Since the tame symbol is surjective, by composing $\tau_{\ppp}$ with the map $\psi$ of Corollary \ref{K2(2K)K2(K)}, we obtain a surjective homomorphism
\[
T_\ppp:= \tau_{\ppp} \circ \psi : K_2(2, K) \rightarrow \kappa(\ppp)^\times.
\]
Thus, for a finite nonempty set $S$ of primes in $K$, there is an induced map
   \begin{align*}
        T_S: K_2(2, K) \rightarrow \bigoplus_{\ppp \notin S} \kappa(\ppp)^\times, \ \text{ with } \ c(u,v) \mapsto \prod_{\ppp \notin S} \tau_{\ppp} (\{u, v\}).
    \end{align*}
\begin{dfn}
    The group $\tilde{K}_2(2, \OO_{K,S})$ is the kernel of the map $T_S$.
\end{dfn}

Let $\Omega$ denote the set of all real embeddings of $K$. Given $\sigma \in \Omega$, consider the following homomorphism:
\begin{align}
\label{K2KF2}
    T_{\sigma}: K_2(K) \rightarrow \F_2, \ \text{ with } \
    \{x, y\}&\mapsto \begin{cases}
        1,  &\text{ if } \operatorname{sign}(\sigma(x)), \operatorname{sign}(\sigma(y))<0\\
        0, &\text{ otherwise} 
    \end{cases}.
\end{align}
Thus, if $r$ is the number of real embeddings of $K$, we can consider the induced homomorphism:
\[
T_{\Omega}:=\oplus_{\sigma \in \Omega} T_{\sigma}: K_2(K) \rightarrow (\F_2)^r.
\]
\begin{dfn}
    We define $K_2(K)_+$ as the kernel of the map $T_{\Omega}$, which is the so called \textit{group of totally positive elements of} $K_2(K)$. Moreover, we define $K_2(\OO_{K,S})_+$ as the intersection $K_2(K)_+ \cap K_2(\OO_{K,S})$.
\end{dfn}

Equivalently, we can say $K_2(\OO_{K,S})_+$ is the kernel of the map $ T_{\Omega}|_{K_2(\OO_{K,S})}: K_2(\OO_{K,S}) \rightarrow (\F_2)^r$. Note that, using the surjectiveness of $T_{\Omega}|_{\OO_K}$ \cite[Lemma 2.4]{keune1989} and analyzing the commutative square
\begin{center}
    \begin{tikzcd}
         K_2(\OO_K) \arrow[r, two heads, "{T_{\Omega}|}"] \arrow[d] & (\F_2)^{r} \arrow[d, equal]\\
         K_2(\OO_{K,S}) \arrow[r, "{T_{\Omega}|}"]  & (\F_2)^{r},
    \end{tikzcd}
\end{center}
we conclude that $T_{\Omega}|_{K_2(\OO_{K,S})}$ is surjective. Thus, we always have a short exact sequence
\begin{equation}\label{seq-k2+}
    0 \rightarrow K_2(\OO_{K,S})_+ \rightarrow K_2(\OO_{K,S}) \rightarrow (\F_2)^{r} \rightarrow 0.
\end{equation}

A theorem of Garland states that $K_2(\OO_K)$ is a finite group \cite{garland1971}. Additionally, it is well known that this group is the tame kernel, and the same can be said about $K_2(\OO_{K,S})$ when considering only the primes not in $S$ \cite[\S 4]{geijsberts}. By applying the Snake Lemma to the commutative diagram of exact rows
\begin{center}
    \begin{tikzcd}
        0 \arrow[r] & K_2(\OO_{K}) \arrow[r] \arrow[d] &  K_2(K) \arrow[r] \arrow[d, equals] & \bigoplus_{\ppp \in \Spec(\OO_K)} \kappa(\ppp)^{\times} \arrow[r]\arrow[d, two heads] & 0\\
        0 \arrow[r] & K_2(\OO_{K, S}) \arrow[r] &  K_2(K) \arrow[r] & \bigoplus_{\ppp \notin S} \kappa(\ppp)^{\times} \arrow[r] & 0,
    \end{tikzcd}
\end{center}
 we conclude that $K_2(\OO_{K,S})$ is finite.  

The next lemma characterizes $\tilde{K_2}(2, \OO_{K,S})$ in terms of $K_2(\OO_{K,S})_+$.

\begin{lem} \cite[Lemma 6.3]{h2016}
\label{K2tilde}
    $\tilde{K_2}(2, \OO_{K,S}) \simeq K_2(\OO_{K,S})_+\oplus \ \z^r$,
    where $r$ denotes the number of real embeddings of $K$. 
\end{lem}

Some interesting examples of $\tilde{K}_2(2, \OO_{K,S})$ are as follows.

\begin{exa}\cite[Example 6.2]{h2016}
\label{irredk2z}
    If $K=\q$ and $m >1$ is an integer, then $\tilde{K_2}(2, \z[1/m])\simeq \z \oplus \bigoplus_{p\mid m} \F_p^{\times}$.
\end{exa}

\begin{exa} \cite[Lemma 6.1]{h2016}
\label{K2Ima1}
    If $K$ is a global field of positive characteristic or a totally imaginary number field, then 
    $\tilde{K_2}(2, \OO_{K,S}) \simeq K_2(\OO_{K,S})$. 
\end{exa}

For next theorem, recall the definition of $\mathcal{S}_2$ given in (\ref{S_2}). This is the main theorem of \cite{h2016}, and shows that $\ker(H_2(\SL_2(\OO_{K,S}), \z) \rightarrow H_2(\SL_2(K), \z))$ is trivial for $S$ big enough. 

\begin{thm}\cite[Theorem 6.10]{h2016}\label{Main2}
   Given a global field $K$, there exists a finite set of primes $S_0 \supseteq \mathcal{S}_2$ such that, for all $S \supseteq S_0$, 
    \[
    H_2(\SL_2(\OO_{K,S}), \z) \simeq \tilde{K_2}(2, \OO_{K,S}).
    \]
    Hence, there is a short exact sequence
    \[ 0 \rightarrow H_2(\SL_2(\OO_{K,S}), \z) \rightarrow H_2(\SL_2(K), \z) \rightarrow \bigoplus_{\ppp \notin S} \kappa(\ppp)^{\times} \rightarrow 0.\]
\end{thm}
 

\begin{rem}
   In Hutchinson's original formulation, he states that there exists a set of primes $S_0$ such that $\OO_{K,S_0}$ contains a unit $\lambda$ with $\lambda^2 -1$ unit, and satisfies the isomorphism with $\tilde{K}_2(2, \OO_{K,S_0})$. We believe that, in general, we can find a set of primes $S_0$ satisfying weaker conditions for which $H_2(\SL_2(\OO_{K,S}), \z) \simeq \tilde{K}_2(2, \OO_{K,S})$, for all $S \supseteq S_0$ (for instance, see Example \ref{Z1/m}). 
   
\end{rem}
\begin{rem}
\label{Sharpness}
     The isomorphism between $H_2(\SL_2(\OO_{K,S}), \z)$ and $\tilde{K}_2(2, \OO_{K,S})$ is not always true. For example, by \cite{an1998}, we have $H_2(\SL_2(\z[1/37]),\z) \simeq \z^6 \oplus \z/6$. But, according to Example \ref{irredk2z}, $\tilde{K_2}(2, \z[1/37])\simeq \z \oplus \F_{37}^\times$. In this case, $2 \mid |H_1(\SL_2(\z[1/37]),\z)| = 12$, suggesting that the hypothesis of $2 \nmid |H_1(\SL_2(\OO_{K,S}), \z)|$ (or $S \supseteq \mathcal{S}_2$) is needed. 
     
\end{rem}


\subsection{The torsion of \texorpdfstring{$H_2(\SL_2(\OO_{K,S}), \z)$}{Lg}}
As we have seen in the previous subsection, there are cases where the second homology of $\SL_2(\OO_{K,S})$ is isomorphic to the group $\tilde{K}_2(2, \OO_{K,S})$. Now, we explore the consequences: provided we have this isomorphism, we can obtain the rank and size of the torsion of the second homology.

\begin{thm} \label{torsionH_2}
    Given a global field $K$, with $r$ real embeddings, there exists a finite set of primes $S_0 \supseteq \mathcal{S}_2$ such that, for all $S \supseteq S_0$,
    \begin{enumerate}
        \item[(i)] $\operatorname{rk}(H_2(\SL_2(\OO_{K,S})), \z)=r$;
        \item[(ii)] $|H_2(\SL_2(\OO_{K,S}), \z)_{\operatorname{tor}}| = |K_2(\OO_{K,S})|/{2^r}$.
    \end{enumerate}
\end{thm}

\begin{proof}
    Let $S_0$ be the set of primes given by Theorem \ref{Main2}. From Lemma \ref{K2tilde}, we know that 
    \[ H_2(\SL_2(\OO_{K,S}), \z) \simeq K_2(\OO_{K,S})_+ \oplus \z^r.\]
    Since $K_2(\OO_{K,S})/K_2(\OO_{K,S})^{+} \simeq (\F_2)^r$ (see (\ref{seq-k2+})), we have that 
    \[|K_2(\OO_{K,S})_+| = \frac{|K_2(\OO_{K,S})|}{2^r}.\]
    We conclude that the rank of the second homology of $\SL_2(\OO_{K,S})$ is $r$ and the size of the torsion is $|K_2(\OO_{K,S})|/2^r$. 
\end{proof}

We remark that the size of the group $K_2(\OO_{K,S})$ has been obtained by numerical methods in many cases, such as for the ring of integers of real quadratic number fields \cite{limqin2025}.




The previous result can be seen through the lens of cohomology. This is done in view of the next lemma, for which we present a proof for completeness (we also refer to \cite[\S 3.2]{PIBT}).

\begin{lem}\label{coh-torsion}
    Let $G$ be a group such that $H_n(G, \z)$ and $H_{n+1}(G, \z)$ are finitely generated. Then, $H^{n+1}(G, \z)$ is also finitely generated and
    \[ \operatorname{rk}H_{n +1}(G, \z) = \operatorname{rk} H^{n+1}(G, \z) \ \ \text{ and } \ \ H_n(G, \z)_{\operatorname{tor}} \simeq H^{n+1}(G, \z)_{\operatorname{tor}}.\]
\end{lem}
\begin{proof}
    The Universal Coefficients Theorem for cohomology \cite[Theorem 12.11]{rotman1988} gives the following split short exact sequence:
    \[
    0 \arr {\rm Ext}_{\z}^1(H_n(G, \z), \z) \arr H^{n+1}(G, \z) \arr {\rm Hom}_{\z}(H_{n+1}(G,\z), \z) \arr 0.
    \]
    Recall some well-known properties \cite{rotman2008}: the functors $\rm Ext_{\z}^1( -, \z)$ and $\rm Hom_{\z}(-, \z)$ commute with direct sums of finitely many summands; $\rm Ext_{\z}^1(\z/m\z, \z) \simeq \z/m\z$; $\rm Hom_{\z}(\z, \z) \simeq \z$; $\rm Ext_{\z}^1(\z, \z)$ and $\rm Hom_{\z}(\z/n\z, \z)$ are trivial.

    By applying these properties, the above short exact sequence boils down to
    \[
    0 \arr H_n(G, \z)_{\operatorname{tor}} \arr H^{n+1}(G, \z) \arr \z^{\operatorname{rk}H_{n +1}(G, \z)} \arr 0,
    \]
    and so the result follows by the splitting.
\end{proof}

Theorem \ref{torsionH_2} then becomes as follows.

\begin{cor} \label{H2structCoh}
    Given a global field $K$, with $r$ real embeddings, there exists a finite set of primes $S_0 \supseteq \mathcal{S}_2$ such that, for all $S \supseteq S_0$, 
    \begin{enumerate}
        \item[(i)] $\operatorname{rk}(H^2(\SL_2(\OO_{K,S})), \z)=r$;
        \item[(ii)] $|H^3(\SL_2(\OO_{K,S}), \z)_{\operatorname{tor}}| = |K_2(\OO_{K,S})|/{2^r}$.
    \end{enumerate}
\end{cor}


\section{A homological version of the Birch-Tate Conjecture}

The aim of this last section is to offer a homological version of Birch-Tate formula for $\OO_{K,S}$ for some $S$ provided that $ 2 \nmid |H_1(\SL_2(\OO_{K,S}), \z)|$. During this section, we assume that $K$ is a totally real extension of $\q$. 


If $s\in \mathbb{C}$ such that $\operatorname{Re}(s)>1$, we define the Dedekind zeta function associated to an algebraic number field $K$ as the convergent series:
\[
\zeta_K(s)=\sum_{\mathfrak{a} \leq \OO_{K}}\frac{1}{N\mathfrak{a}^s}. 
\]
From Tate's thesis \cite{tate1950}, we know that $\zeta_K(s)$ has analytic continuation. One interesting question that we may ask is about the values of $\zeta_K(m)$, for integer values of $m$, but this is in general a very tough task, even deciding whether this is a rational number or not. For $m<0$ we have the Sielgel-Klingen theorem.

\begin{thm} {\cite{klingen1962}}
    Let $m$ be a negative integer. Then $\zeta_K(m) \in \q.$
\end{thm}

For a totally real number field $K$, the famous Birch-Tate conjecture suggests that there is a relation of $K_2(\OO_{K})$ with $\zeta_K(-1)$, as announced follows: 

\begin{conj}[Birch-Tate Conjecture] \label{BTC}
    Let $K$ be a totally real number field. Then, up to a power of $2$,
     \[
     |\zeta_K(-1)|=\frac{|K_2(\OO_{K})|}{|w_2(K)|},
     \] 
     where  $w_2(K)$ be the number of roots of unity on $K(\sqrt{K})$.
\end{conj}

For abelian extensions, this conjecture is known to be true as a consequence of seminal works of Mazur and Wiles {\cite{MW1984, wiles1990}}, where the Iwasawa main conjecture is proven.

\begin{thm}  \label{BTT}
Let $A=\OO_K$ be the ring of integers of a totally real abelian extension. Then
 \[
|\zeta_K(-1)|=\frac{|K_2(\OO_{K})|}{|w_2(K)|}.
\]
\end{thm}

Later, Birch-Tate conjecture was extended by Lichtenbaum \cite{Lich1972}, relating even $K$-groups with odd values of Dedekind zeta functions. As in the case of Birch-Tate conjecture, this conjecture was proved for abelian extension, on which we can cite the work of Rost-Voevodsky \cite{Voevodsky2011}.

Following \cite{sands2006}, there is a general version of Birch-Tate conjecture over Dedekind $\zeta$ functions associated to a ring of $S$-integers $\OO_{K,S}$. For $s \in \mathbb{C}$, with $\operatorname{Re}(s) > 1$,
\[ \zeta_K^{S}(s) =\sum_{\mathfrak{a} \leq \OO_{K,S}}\frac{1}{N\mathfrak{a}^s},\] 
and we extend the function by analytic continuation. This leads to a Birch-Tate conjecture for ring of $S$-integers. 

\begin{conj}
    Let $\OO_{K,S}$ be a ring of $S$-integers. Then, up to a power of $2$, 
    \[
    |\zeta_K^{S}(-1)|= \frac{|K_2(\OO_{K,S})|}{|w_2(K)|}.
    \]
\end{conj}

We now present a homological version of the Birch-Tate formula inspired by the results developed in the last section. As immediate consequences of Theorem \ref{torsionH_2}, we restate the Birch-Tate conjecture in terms of $H_2(\SL_2(\OO_{K, S}), \z)$ and the Birch-Tate formula, assuming that there is no prime $\ppp$ over $2$, with $|\kappa(\ppp)|=2$, i.e. $2 \nmid |H_1(\SL_2(\OO_{K,S}), \z)|$. 

\begin{conj}
\label{BThomconj}
    Let $\OO_{K,S}$ be a ring of $S$-integers of a totally real extension $K$ with $[K: \q] = r$. Then, there is an $S_0 \supseteq \mathcal{S}_2$ such that, for all $S \supseteq S_0$, we have, up to a power of $2$,
    \[
    |\zeta_K^S(-1)|=\frac{2^r \cdot |H_2(\SL_2(\OO_{K,S}), \z)_{\operatorname{tor}}|}{|w_2(K)|}.
    \]
\end{conj}

Then, Theorem \ref{BTT} can be restated in homological version as follows. 

\begin{thm}
\label{BThomthm}
    Let $\OO_{K,S}$ be a ring of $S$-integers of a totally real extension $K$ with $[K: \q] = r$. Then, there is an $S_0 \supseteq \mathcal{S}_2$ such that, for all $S \supseteq S_0$, we have, up to a power of $2$,
    \[
    |\zeta_K^S(-1)|=\frac{2^r \cdot |H_2(\SL_2(\OO_{K,S}), \z)_{\operatorname{tor}}|}{|w_2(K)|}.
    \]  
\end{thm}


\end{document}